\documentclass[11pt,onecolumn]{article}

\usepackage{amsmath}
\usepackage{amssymb}
\usepackage{amsthm}
\usepackage{graphicx}
\usepackage{mathtools}
\usepackage{authblk}
\usepackage{hyperref}
\usepackage{cleveref}

\usepackage{tikz-cd}

\theoremstyle{plain}
\newtheorem{theorem}{Theorem}[section]
\newtheorem{corollary}[theorem]{Corollary}

\theoremstyle{definition}
\newtheorem{definition}[theorem]{Definition}
\newtheorem{example}[theorem]{Example}
\newtheorem{remark}[theorem]{Remark}

\newcommand{\C}{\mathbb{C}}
\newcommand{\R}{\mathbb{R}}
\newcommand{\N}{\mathbb{N}}

\newcommand{\littletaller}{\mathchoice{\vphantom{\big|}}{}{}{}}
\newcommand\restr[2]{{
		\left.\kern-\nulldelimiterspace 
		#1 
		\littletaller 
		\right|_{#2} 
}}

\title{Lifts of Operator Systems}
\author{Markus Dannem\"uller}
\author{Tim Netzer}
\affil{\small Department of Mathematics, University of Innsbruck, Austria}
\date{\today}

\begin{document}
\maketitle
\begin{abstract}
We transfer the theory  of slack operators and sums-of-squares-criteria for lifts from   convex cones to operator systems. These allow to study the following question, among others: Given an abstract operator system, is its enveloping $C^*$-algebra  finite-dimensional, or is it the linear image of  a system with a finite-dimensional enveloping algebra?  
\end{abstract}

\section{Introduction}

The study of convex sets and their representations through more tractable objects, such as spectrahedra or free spectrahedra, lies at the heart of modern convex optimization and real algebraic geometry. A central question is whether a given convex set can be represented as a linear image of an affine section of a cone of positive semidefinite matrices. Such representations, known as semidefinite representations or spectrahedral lifts, are highly desirable due to the availability of efficient algorithms for semidefinite programming.

In the classical (commutative) setting, foundational work from \cite{GPT, yann} established a deep connection between the existence of such lifts and the factorization of slack operators. This line of research was further enriched by sum-of-squares characterizations, which link the existence of a lift to the ability to express all positive linear functionals as sums of squares with respect to a finite-dimensional function space \cite{faw, fgpst, schei}.

This paper extends this framework to the noncommutative setting of operator systems. Operator systems provide a natural generalization of convex cones to the matrix-valued context, capturing the structure of positivity across all matrix levels. They play a fundamental role in operator theory and quantum information.

We introduce and study noncommutative lifts and factorizations for operator systems, mirroring and generalizing the classical notions. Our main results show that an operator system admits a  lift to another operator system $T$ if and only if its (noncommutative) slack operator admits a $T$-factorization. This extends the main theorem from \cite{GPT} to the operator system setting. We further  develop a sum-of-squares characterization for lifts of minimal operator systems -- those whose positivity structure is entirely determined by a base cone. This allows us to derive obstructions to the existence of free spectrahedral lifts using tools from real algebraic geometry, including the construction of positive matrix polynomials that are not sums of hermitian squares.

Our work not only unifies and extends existing lift-and-factorization theories but also opens new pathways for studying finite-dimensionality and representability in operator systems. The results have implications for the structure of enveloping $C^*$-algebras and the representation of convex sets in quantum and noncommutative optimization.

\section{Preliminaries}

\subsection{Lifts and Factorizations for Convex Cones} \label{sec:lift_fact}

In this section, we collect  important definitions and results on  slack operators,  sum-of-squares certificates for positive functionals, and  lifts for convex cones.  In \Cref{sec:results} we will transfer most of these concepts and results to the non-commutative setup of operator systems.

All vector spaces in this paper are real and equipped with the finest locally convex topology, all topological notions refer to it.
A convex cone \(C \subseteq V\) in a vector space \(V\) is called \textit{proper} if it is \textit{pointed} (i.e.\  $C \cap -C = \lbrace 0 \rbrace$), closed and has non-empty interior.
Given a convex cone \(C\), we call a subset a \textit{generating set} for \(C\), usually denoted \(\chi(C)\), if \(C = \operatorname{cc}(\chi(C))\), where \(\operatorname{cc}(\cdot)\) denotes the conic hull.

For a given convex cone \(C \subseteq V\), one can ask whether it is representable in terms of a different convex  cone. This is of interest in particular when the other cone is in some sense more well-behaved than \(C\). More precisely, given a convex cone $K\subseteq W$, we say that \(C\) {\it admits a $K$-realization}, if $C\cong H\cap K$ for some subspace $H\subseteq W$. The realization is {\it proper}, if $H$ intersects the interior of $K$. $C$ is said to have a \textit{\(K\)-lift}, if it is the linear image of a cone that has a $K$-realization.   The lift is called {\it proper} if the corresponding realization is proper.

A particularly interesting family of such lifts are those where  \(K = \operatorname{Psd}_d(\mathbb C)\) is a cone of positive semidefinite hermitian matrices. In this case, a cone with a $K$-realization is also called  a \textit{spectrahedron}, and a cone with a $K$-lift a \textit{spectrahedrop}.

It was shown in \cite{GPT} that the existence of \(K\)-lifts has different interesting characterisations. We note at this point that many of the following definitions and results are often formulated for bounded convex sets, but they can be  adapted easily to the  conic case.
First, one defines the notion of a slack operator and of a \(K\)-factorization:

\begin{definition}
	Let \(C\) be a proper convex cone with dual cone \(C^\lor\subseteq V'\), and let \(\chi(C)\) and \(\chi(C^\lor)\) be generating sets for \(C\) and \(C^\lor\), respectively. 
	The map \[\langle\cdot,\cdot\rangle_C\colon C \times C^\vee\to \R_{\geqslant 0};\ (c,f)\mapsto f(c)\] is called the \textit{slack operator} of \(C\).
	A pair of (possibly non-linear) maps \[\alpha\colon C \to K, \quad \beta\colon C^\lor \to K^\lor\] is called \(K\)\textit{-factorization} for the slack operator of $C$, if for all \(c \in \chi(C)\) and for all \(f \in \chi(C^\lor)\) it holds that $$\langle c,f\rangle_C = \langle \alpha(c),\beta(f)\rangle_K.$$
\end{definition}

One of the main results in \cite{GPT} is the following theorem, which shows that the existence of \(K\)-lifts and \(K\)-factorizations is essentially equivalent, with one direction requiring a mild additional assumption:

\begin{theorem} \label{thm:lift_iff_fact}
	Let \(C\) and \(K\) be finite-dimensional proper convex cones. If \(C\) has a proper \(K\)-lift, then its slack operator has a \(K\)-factorization. Conversely, if the slack operator of $C$ has a \(K\)-factorization, then \(C\) has a \(K\)-lift.
\end{theorem}

Since we will be mostly interested in spectrahedral lifts, the following sum-of-squares formulation of the existence of a \(\operatorname{Psd}_d(\mathbb C)\)-lift  is also of interest. The sufficient condition for a lift goes back to \cite{las}, the necessity was first shown in \cite{schei}, and slightly adapted in \cite{faw} (see also \cite{fgpst, nepl} for more details). 

\begin{theorem} \label{thm:lift_iff_SOS}
	Let \(C\) be a finite-dimensional proper convex cone and let \(W \subseteq \mathcal{F}(C, \R)\) be a finite-dimensional subspace of functions such that for any \(f \in C^\lor\), we find a sum-of-squares representation with elements \(h_k \in W\), i.e. \[f(c) = \sum_k h_k(c)^2\] for all \(c \in \chi(C)\). Then \(C\) admits a \(\operatorname{Psd}_d(\mathbb C)\)-lift for \(d = \operatorname{dim}(W)\). Conversely, if \(C\) admits a \(\operatorname{Psd}_d(\mathbb C)\)-lift, then there exists a subspace \(W\) of functions with the above property and with \(\operatorname{dim}(W) \leqslant d^2\).
\end{theorem}

\subsection{Operator Systems}
In this section we collect some basic facts about operator systems, see \cite{pau} for more detailed information and references.
An \textit{abstract operator system} $S$ on the vector space $X$ consists of a proper convex cone $$S_s\subseteq X\otimes {\rm Her}_s(\mathbb C)$$ for each $s\geqslant 1,$ such that for all $V\in{\rm Mat}_{s,t}(\mathbb C)$ the linear map $${\rm id}_X\otimes (V^*\cdot V) \colon X\otimes {\rm Her}_s(\mathbb C)\to X\otimes{\rm Her}_t(\mathbb C)$$ is positive with respect to these cones, i.e.\ maps $S_s$ into $S_t$.  Here we denote by ${\rm Her}_s(\mathbb C)$ the vector space of complex Hermitian matrices.

For each Hilbert space $H$ we denote by $\mathbb B_{\rm sa}(H)$ the space of self-adjoint bounded operators, and we use the identification $$\mathbb B_{\rm sa}(H)\otimes{\rm Her}_s(\mathbb C)=\mathbb B_{\rm sa}(H^s)$$ to define an important  operator system $P_H$ on the space $\mathbb B_{\rm sa}(H)$, defined as $$(P_H)_s\coloneqq \mathbb B_{\rm psd}(H^s),$$ the cone of positive operators. This is also known as the usual operator system structure on the $C^*$-algebra $\mathbb B(H)$. In case $H=\mathbb C^d$ we also write $P^d$ instead of $P_{\mathbb C^d}$ for this system on the space ${\rm Her}_d(\mathbb C)$.

If $S,T$ are abstract operator systems on spaces $X$ and $Y$, respectively,  a linear map $\phi\colon X\to Y$ is \emph{completely positive}, if all maps $$\phi\otimes{\rm id}_{{\rm Her}_s(\mathbb C)}\colon X\otimes{\rm Her}_s(\mathbb C)\to Y\otimes{\rm Her}_s(\mathbb C)$$ are positive, i.e.\ map $S_s$ into $T_s.$ We denote by ${\rm CP}(S,T)$ the convex cone of all completely positive maps from $S$ to $T,$ which lives in the space ${\rm Lin}(X,Y)$ of all linear maps from $X$ to $Y$.

The \textit{dual} of an abstract operator system \(S\) on $X$ is the operator system \(S^\lor\) on $X'$ defined by 
$$(S^\vee)_s\coloneqq (S_s)^\vee.$$  Note that this might fail to fulfill the nonempty interior condition for the cones, but by slight abuse of notation  we will still call it an operator system. Using the canonical identification 
 $X'\otimes{\rm Her}_s(\mathbb C)={\rm Lin}(X,{\rm Her}_s(\mathbb C))$ it can be checked that $$(S^\vee)_s={\rm CP}(S,P^{s})$$ holds. 

\section{Lifts and Factorizations for Operator Systems}
\label{sec:results}

For convenience, we introduce some more notation. Let \(X\) and \(Y\) be vector spaces, \(\phi\colon X \to Y\) a linear map, and \(S, T\) operator systems on \(X\) and \(Y\), respectively. Then for \(a \in X\otimes{\rm Her}_s(\C)\) we define $$\phi[a] \coloneqq (\phi \otimes {\rm id}_{{\rm Her}_s(\mathbb C)})(a).$$ Similarly, for \(b \in Y\otimes{\rm Her}_s(\C)\) we define \(\phi^{-1}[b]\subseteq X\otimes{\rm Her}_s(\C)\) by \[a \in \phi^{-1}[b] :\Leftrightarrow \phi[a] = b.\]
With this, we are ready to define non-commutative analogues for the lifts and factorizations from \Cref{sec:lift_fact}.

\begin{definition}
	Let \(X, Y, Z\) be vector spaces, $\psi\colon X\to Y, \pi\colon Z \to X$ and $\gamma\colon Z \to Y$ linear maps, and \(S\) and \(T\) operator systems on \(X\) and \(Y\), respectively.
    Then $\psi$ is a {\it $T$-realization} of $S,$ if $S=\psi^{-1}[T],$ where the inverse image is understood at each level $s$. We call the realization {\it proper}, if $\psi$ hits the interior of $T_1$ in $Y$.
    We say that the pair \((\pi, \gamma)\) is a {\it $T$-lift} of \(S\) if $\gamma$ is injective and \[S = \pi[\gamma^{-1}[T]],\] again levelwise.
    The lift is \textit{proper}, if $\gamma$ hits the interior of $T_1.$
\end{definition}

\begin{remark} \label{rem:lift}
    If $\gamma$ is proper, then \(\gamma^{-1}[T]\) is indeed an operator system. Just like in the commutative case in \Cref{thm:lift_iff_fact}, this notion of properness will be required for one direction of an analogous free statement.
\end{remark}

\begin{remark}
    The Choi-Effros Theorem states that every operator system has a proper $P_H$-realization, for some Hilbert space $H$. The fact that $H$ can be chosen finite-dimensional is equivalent to the enveloping $C^*$-algebra of the operator system being finite-dimensional, which is clear from the universal property. In view of the above, such systems are also called {\it free spectrahedra,} and the systems with a lift to such a system are called {\it free spectrahedrops.}
\end{remark}

Next, we define the notion of a free slack operator. For this, the evaluation map from the commutative case is replaced with a more general bilinear evaluation map.

\begin{definition}
    For an operator system \(S\) with dual \(S^\lor\), we define its \textit{slack operator} as the family of maps \[[\cdot,\cdot ]_{S,s,t}\colon S_s \times S^\lor_t \to \operatorname{Psd}_{st}(\mathbb C);\  (a, \phi) \mapsto \phi[a]\] where we use the identification
    $S^\vee_t={\rm CP}(S,P^t).$  We will typically suppress the indices \(s, t\) and write $[\cdot,\cdot]_S$ for any of the maps.
\end{definition}

Next, we define an analogue of factorizations for slack operators:

\begin{definition}
	Given operator systems \(S\) and \(T\), a {\it $T$-factorization} of the slack operator $[\cdot,\cdot]_S$ of $S$ consists of maps \[\alpha_s\colon S_s \to T_s, \qquad \beta_t\colon S_t^\lor \to T_t^\lor,\] not necessarily linear,  such that for all \(s, t \in \N\), \(a \in S_s\) and \(\phi \in S^\lor_t\) we have \[ [a, \phi]_S = [\alpha_s(a),\beta_t(\phi)]_T .\] 
\end{definition}

With these definitions, we are now ready to formulate our first main result, a natural adaptation of Theorem \ref{thm:lift_iff_fact} to the setting of operator systems. The proof is similar in spirit to the one from \cite{GPT}, but we have to use Arveson's Extension Theorem instead of Farkas' Lemma in a crucial step.

\begin{theorem} \label{thm:NC_lift_iff_fact}
	Let \(S\) and \(T\) be operator systems on \(X\) and \(Y\), respectively. If \(S\) admits a proper \(T\)-lift, then its slack operator  admits a \(T\)-factorization. Conversely, if the slack operator of $S$ admits a \(T\)-factorization, then \(S\) admits a  \(T\)-lift.
\end{theorem}
\begin{proof}
	Let us assume first that \([\cdot,\cdot]_S\) admits a \(T\)-factorization, given by the family of maps \((\alpha, \beta)\). We define the vector space \[Z \coloneqq \lbrace (x, y) \in X \oplus Y \mid \forall t, \phi \in S^\lor_t\colon \phi(x) = \beta_t(\phi)(y) \rbrace,\] which comes equipped with the canonical projections \[\pi_X\colon Z \to X, \qquad \pi_Y\colon Z \to Y.\] Our goal is to show that \((\pi_X, \pi_Y)\) is a \(T\)-lift for \(S\). Hence, we need to show that $\pi_Y$ is injective and \(\pi_X[\pi_Y^{-1}[T]] = S\) holds. 

   For injectivity of $\pi_Y$, assume \((x, y) \in \operatorname{ker}(\pi_Y)\), so clearly \(y = 0\). But then \((x, 0)\) is such that \(\forall t, \phi \in S^\lor_t\) we have \(\phi(x) = \beta_t(\phi)(0) = 0\), since \(\beta_t(\phi)\) is a linear map. Already from the case $t=1$ we get  $x\in S_1\cap -S_1=\{0\}$, proving injectivity. 
	
    Now let \(c \in \pi_Y^{-1}[T]_s\), \(\phi \in S^\lor_s\), and consider \(\phi[\pi_X[c]]\in{\rm Her}_{s^2}(\C).\) By construction, the maps \(\phi \circ \pi_X\) and \(\beta_s(\phi) \circ \pi_Y\) coincide on \(Z\). But then they  also agree when applying them to elements of \(\pi_Y^{-1}[T]_s\), so we get \[\phi[\pi_X[c]] = \beta_s(\phi)[\pi_Y[c]].\] Since \(\pi_Y[c] \in T_s\) and \(\beta_s(\phi) \in T^\lor_s\), the right-hand side is positive semi-definite, and this in particular implies  $\phi(\pi_X[c])\geqslant 0.$ Since this holds for every \(\phi \in S^\lor_s\), we get $\pi_X(c)\in S_s$, by classical biduality. This shows that \(\pi_X[\pi_Y^{-1}[T]] \subseteq S\).
	
	For the reverse inclusion, let \(a \in S_s\). Since \(\alpha_s(a) \in T_s\), we clearly have $\pi_Y^{-1}[\alpha_s(a)] \subseteq \pi_Y^{-1}[T]_s.$ We now define $$\tilde{a} \coloneqq (a, \alpha_s(a))\in (X\otimes{\rm Her}_s(\C))\oplus(Y\otimes{\rm Her}_s(\C))=(X\oplus Y)\otimes{\rm Her}_s(\C).$$ Once we have shown that $\tilde a\in Z\otimes{\rm Her}_s(\C)$ holds, we are done, since  \(\tilde{a} \in \pi_Y^{-1}[\alpha_s(a)]\subseteq \pi_Y^{-1}[T]_s,\) and $\pi_X[\tilde a]=a$ holds. To prove that $\tilde a$ lives over $Z$, let us fix a basis \((M_i)_i\) of \(\operatorname{Her}_s(\C)\) and write \[a = \sum_i x_i \otimes M_i, \qquad \alpha_s(a) = \sum_i y_i \otimes M_i.\]  From the definition of a factorization we obtain \begin{align*}\sum_i\beta_t(\phi)(y_i)\otimes M_i&=\beta_t(\phi)[\alpha_s(a)]=\phi[a]=\sum_i\phi(x_i)\otimes M_i.\end{align*} Since the $M_i$ form a basis, this implies $\beta_t(\phi)(y_i)=\phi(x_i)$ for all $i$, i.e.\ $(x_i,y_i)\in Z.$ Now clearly $\tilde a=\sum_i (x_i,y_i)\otimes M_i\in Z\otimes{\rm Her}_s(\C)$, as desired.
	
	Taking both inclusions together, we conclude that \(S = \pi_X[\pi_Y^{-1}[T]]\), hence \((\pi_X, \pi_Y)\) forms a \(T\)-lift for \(S\).
	
	For the converse statement, let us assume that \(S\) admits a proper \(T\)-lift, given by  linear maps \(\pi\colon Z \to X\) and \(\gamma\colon Z \to Y\) for some vectorspace \(Z\), such that $\gamma$ hits the interior of $T_1$. 
	
	We note that for any \(a \in S_s\), the fibre \(\pi^{-1}[a]\) has non-empty intersection with \(\gamma^{-1}[T]\). So we can just pick an arbitrary  element \(\tilde{a} \in \pi^{-1}[a] \cap \gamma^{-1}[T]\) and define \(\alpha_s(a) \coloneqq \gamma[\tilde{a}]\in T_s.\)
	
	Next we need to construct \(\beta_t(\phi)\in T^\vee_t\) for any given \(\phi \in S^\lor_t\). Understanding again \(\phi\) as an element of \({\rm CP}(S, P^t)\), we pull this map back along \(\pi\) and define \[\tilde{\phi} \coloneqq \phi \circ \pi.\] For any \(c \in \gamma^{-1}[T]\) we have \[\tilde{\phi}[c] = \phi[\pi[c]],\] and since \(\pi[c] \in S\) and \(\phi \in S^\lor\), this expression is positive semidefinite. This shows that \[\tilde{\phi} \in {\rm CP}(\gamma^{-1}[T], P^t).\]
	We now  extend \(\tilde{\phi}\) to a completely positive map on \(T\), which we then take as \(\beta_t(\phi)\). This  makes sense since \(\gamma\) is injective, and we can thus view it as the inclusion of \(\gamma^{-1}[T]\) into \(T\), and  Arveson's Extension Theorem guarantees the existence of an extension (which uses properness of $\gamma$).   We now have  \[[\alpha_s(a),\beta_t(\phi)]_T=\beta_t(\phi)[\alpha_s(a)] = \tilde{\phi}[\tilde{a}] = \phi[\pi[\tilde{a}]] = \phi[a]=[a,\phi]_S,\] which shows that \((\alpha, \beta)\) is indeed a \(T\)-factorization of the slack operator of \(S\). 
\end{proof}

When looking carefully at the construction of the \(T\)-lift from a given factorization \((\alpha,\beta)\), we notice that we do not really need to have \(\alpha\) and \(\beta\) defined on all of \(S\) and \(S^\lor\), respectively. Instead, it is sufficient to define them on generating sets \(\chi(S)\) and \(\chi(S^\lor)\) which generate \(S\) and \(S^\lor\) as operator systems, i.e. using matrix-conic combinations. However, since neither \(\alpha\) nor \(\beta\) are assumed to be linear, we can simply extend them to all of \(S\) and \(S^\lor\) by fixing for each element a (non-unique) matrix conic combination of the generators and extending \(\alpha\) and \(\beta\) linearly on this combination. So in this sense, it is irrelevant whether we define a factorization on all of \(S\) and \(S^\lor\) or only on generating sets.

If \(C\) is a proper cone and \(S = C^\textnormal{min}\) is the smallest operator system with base cone \(S_1=C\), then it is generated by any generating set \(\chi(C)\) of \(C\). Hence, the factorization map \(\alpha\) only needs to be defined on \(\chi(C),\) and not on levels \(s \geqslant 2\). This means that the sequence of maps \(\alpha\) actually simplifies to a single map \(\alpha_1\colon \chi(C) \to T_1\).

\begin{corollary}
\label{cor:minlift}
	Let \(C\subseteq X\) be a proper cone, and \(T\) an operator system on the space $Y$. Assume there exists a (not necessarily linear) map \(\alpha\colon \chi(C) \to T_1,\) such that for every \(t\geqslant 1\) and every positive linear map \(\phi \in \operatorname{Pos}(C, \operatorname{Psd}_t(\C))\) there exists a map \(\psi \in {\rm CP}(T, P^t)\) with \[\phi(c) = \psi(\alpha(c))\] for all \(c \in \chi(C)\). Then \(C^\textnormal{min}\) admits a \(T\)-lift.
\end{corollary}
\begin{proof}
    This is a direct consequence of \Cref{thm:NC_lift_iff_fact}, using that positivity and complete positivity coincide  on minimal systems, and by setting $\beta_t(\phi)\coloneqq \psi.$
\end{proof}

We can also characterize realizations instead of lifts in terms of factorizations.

\begin{definition}
    A $T$-factorization $(\alpha,\beta)$ of the slack operator of $S$ is called {\it linear}, if $$\alpha_s=\psi\otimes{\rm id}_{{\rm Her}_s(\C)}$$ holds for all $s\geqslant 1,$ where $\psi\in{\rm CP}(S,T)$ is a fixed cp-map. 
\end{definition}

\begin{theorem}\label{thm:factlift2}Let $S,T$ be operator systems. 
If the slack operator of $S$ has a linear $T$-factorization, then $S$ has a $T$-realization.
Conversely, if $S$ has a proper $T$-realization, then the slack operator of $S$ has a linear $T$-factorization.
\end{theorem}
\begin{proof}
First let $\psi\in{\rm CP}(S,T)$ be such that $\alpha_s=\psi\otimes{\rm id}_{{\rm Her}_s(\mathbb C)},$ together with maps $\beta_t,$ provide a linear $T$-factorization for the slack operator of $S$. We will show that $\psi^{-1}[T]=S$ holds, where ``$\supseteq$" is clear from complete positivity of $\psi$. Now if $a\in X\otimes{\rm Her}_s(\mathbb C)\setminus S_s,$ there is some $\phi\in S^\vee_s$ with $\phi(a)<0.$ This implies that $\phi[a]=\beta_s(\phi)[\psi[a]]$ is not positive semidefinite. From $\beta_s(\phi)\in T^\vee_s$ we see that $\psi[a]\notin T_s,$ and thus $a\notin\psi^{-1}[T]_s$, which proves the first claim. 

The second statement follows immediately from the proof of \Cref{thm:NC_lift_iff_fact}, since $\alpha_s=\gamma\otimes{\rm id}_{{\rm Her}_s(\mathbb C)}$ holds, if no projection map $\pi$ is involved. 
\end{proof}

\begin{corollary}
\label{cor:min}
    An operator system \(S\) is a free spectrahedron of size \(d\) if and only if its slack operator  admits a linear \(P^d\)-factorization. For a proper cone $C$, this is true for $C^\textnormal{min}$  if and only if the map $\alpha$ from \Cref{cor:minlift} can be chosen linear.\qed
\end{corollary}

\begin{example} \label{ex:simplex_hedron}
Let \(C \subseteq \R^n\) be a simplex cone, so without loss of generality \(C\) is the positive orthant \(\R_{+}^n\). We want to show that \(C^\textnormal{min}\) is a free spectrahedron of size \(n\). By \Cref{cor:min} we can equivalently find a positive linear map \(\alpha\colon \R^n \to \operatorname{Her}_n(\C)\) which fulfils the factorization condition. We claim that defining \(\alpha\) by \(\alpha(e_i) = E_{ii} = e_i e_i^*\) and then linearly extending yields a linear factorization for the slack operator. 

Let \(\phi \in \left(C^\textnormal{min}\right)^\lor_t \cong \operatorname{Pos}(C, \operatorname{Psd}_t(\C))\). Since the \(e_i\) form a basis and \(\phi\) is linear and positive on each of them, \(\phi\) is uniquely determined by a choice of \(P_1, \dots, P_n \in \operatorname{Psd}_t(\C),\) where \(\phi(e_i)=P_i\). We decompose each of these into rank-1 matrices as \[P_i = \sum_{k=1}^t q_{ik} q_{ik}^*\] for \(q_{ik} \in \C^t\) and then set \[V_{(l, k)} = e_l q_{l k}^*\] with the multi-index \((l, k)\) running over \(l = 1, \dots, n\) and \(k = 1, \dots, t\). For \(i = 1, \dots, n\) we get \[\sum_{(l ,k)} V_{(l, k)}^* \alpha(e_i) V_{(l, k)} = \sum_{(l, k)} q_{lk} e_l^* e_i e_i^* e_l q_{lk}^* = \sum_k q_{ik} q_{ik}^* = P_i =\phi(e_i).\] Hence, \(\alpha\) is a linear factorization map for the slack operator, and thus \(C^\textnormal{min}\) is a free spectrahedron. 
\end{example}

\begin{example}
Let \(C \subseteq \R^n\) be a polyhedral cone with \(m > n\) extreme rays. We show that \(C^\textnormal{min}\) is a free spectrahedrop from dimension  \(m\).

We choose generators \(c_1, \dots, c_m \in \R^n,\) of which we assume w.l.o.g.\ that $c_1,\ldots, c_n$ are a basis of $\R^n$.
We define \(\alpha(c_i) = E_{ii}\in{\rm Her}_m(\C)\) for \(i = 1, \dots, m\). This map is obviously not linear, since its image contains \(m > n\) many linearly independent elements.  A positive map \(\phi\in \operatorname{Pos}(C, \operatorname{Psd}_t(\C))\) is again uniquely determined by a tuple of positive semi-definite matrices \(P_1, \dots, P_m\) with \(\phi(c_i) = P_i\) and such that whenever \(c_j = \sum_{i=1}^n \lambda_i c_i\), then \(P_j = \sum_{i=1}^n \lambda_i P_i\). Just like in 
\Cref{ex:simplex_hedron}, we set \(V_{(l, k)} = e_l q_{lk}^*\) where \(P_i = \sum_k q_{ik} q_{ik}^*\) and by direct calculation see that the factorization condition holds. Since \(\alpha\) is not linear, we can only conclude that \(C^\textnormal{min}\) is a free spectrahedrop. 

We also note the following: While it is known that \(C^\textnormal{min}\) is in fact not a free spectrahedron \cite{fnt}, we would need to show that no linear factorization exists in order to conclude the same.
\end{example}

\section{Lifts of Operator Systems and Sums of Squares}

Next, we want to transfer the result from  \Cref{thm:lift_iff_SOS} to the realm of operator systems. The following theorem does this for the case of minimal operator systems. Also here the proof is similar in spirit to the ones in \cite{faw,schei}, but adapted to the matrix-valued setup.

\begin{theorem} \label{thm:NC_lift_iff_SOHS}
	For a proper convex cone $C$, if the minimal operator system \(C^\textnormal{min}\)
	 admits a \(P^d\)-lift, then there exists a space \(W\)  spanned by  real-valued semialgebraic functions on \(\chi(C)\) with \(\dim(W) \leqslant d^2,\) such that \[\forall t\ \forall \phi \in {\rm Pos}(C,{\rm Psd}_t(\C)) \ \exists P_k \in \operatorname{Mat}_{d,t}(W)\colon \phi(c) = \sum_k P_k(c)^* P_k(c)\] holds for all \(c \in \chi(C)\).
	
	Conversely, if a \(d\)-dimensional space \(W\) of arbitrary complex-valued functions on \(\chi(C)\) with the above sum-of-squares property exists, then \(C^\textnormal{min}\) admits a \(P^d\)-lift.
\end{theorem}

\begin{proof}
	First assume that \(C^\textnormal{min}\) admits a \(P^d\)-lift, i.e.\ there are  maps $\gamma\colon Z\to {\rm Her}_d(\C), \pi\colon Z\to X$ with $C^{\rm min}=\pi[\gamma^{-1}(P^d)].$  So for  $S\coloneqq \gamma^{-1}[P^d]$ we have  $C^{\min}=\pi[S].$   
 
    We can choose a semialgebraic function $\iota\colon C\to S_1$ with $c=\pi(\iota(c))$ for all $c\in C$.  Further, for $z\in S_1$ the matrix $\gamma(z)$ is positive semidefinite, and we denote its (unique) positive semidefinite  square-root by \(Q(z)\). Note that the function $z\mapsto Q(z)$ is semialgebraic. Thus the entries of the $d\times d$ Hermitian matrix $Q(\iota(c))$ can be described by $d^2$-many real valued semialgebraic functions on $C$, which span a space $W$ of dimension at most $d^2$.

    Now let \(\phi \in {\rm Pos}(C,{\rm Psd}_t(\C))={\rm CP}(C^{\rm min},P^t)\). Then  \(\tilde{\phi} \coloneqq \phi \circ \pi\in{\rm CP}(S,P^t).\)  Hence by the Choi-Krauss Theorem  we can find matrices \(V_k\in{\rm Mat}_{d,t}(\C)\) such that  \ \[\phi(\pi(z)) = \tilde{\phi}(z) = \sum_k V_k^* \gamma(z) V_k\] for all $z\in Z$.
     Defining \[P_k(c) \coloneqq Q(\iota(c)) V_k\] yields the desired sums-of-squares decomposition, since the entries of $P_k$ are linear combinations of the $d^2$-many functions from above, and thus lie in $W.$
	
	For the converse, let $W = {\rm span}_\C\{f_1, \dots, f_d\}$ be a space of functions on $\chi(C)$ with the sums-of-squares property. We will show that the conditions of \Cref{cor:minlift} are fulfilled. 
 
    Define \[f \coloneqq \left(f_1, \dots, f_d\right)\] and set \(\alpha(c) \coloneqq f(c)^*f(c)\). Now let $\phi \in {\rm Pos}(C,{\rm Psd}_t(\C))$. By assumption, we find \(P_k \in \operatorname{Mat}_{d,t}(W)\) such that \(\phi(c) = \sum_k P_k(c)^* P_k(c)\) for all $c\in\chi(C)$. Since each entry \([P_k]_{rs}\) is an element of \(W\), we find a complex vector \(b_{krs}\in\C^d\) with \[[P_k]_{rs} = fb_{krs}.\]   Now for \(c \in \chi(C)\) we have
    \begin{align*}[\phi(c)]_{ij} &= \sum_{k,r} [P_k(c)^*]_{ir}[P_k(c)]_{rj} = \sum_{k,r} b_{kri}^*f(c)^*f(c)b_{krj}\\ &=\sum_{k,r} b_{kri}^*\alpha(c)b_{krj}.\end{align*}
    We now define $\psi\in{\rm CP}(P^d,P^t)$ as the map \[X \mapsto \sum_{k,r} V_{kr}^* X V_{kr},\] where $V_{kr}\in{\rm Mat}_{d,t}(\C)$ has $b_{kr1},\ldots, b_{kr
     t}$ as its columns, and obtain $\phi(c)=\psi(\alpha(c)),$ as desired. So \Cref{cor:minlift} yields a $P^d$-lift of $C^{\rm min}.$
\end{proof}

We first emphasise that, just like in \Cref{thm:lift_iff_SOS}, this theorem does not quite allow for a perfect converse. Given a lift of size \(d\), we are guaranteed to find a (at most) \(d^2\)-dimensional space of semi-algebraic functions with the sums-of-squares property. But the converse is slightly different, since it assumes the existence of a \(d\)-dimensional space of (possibly not semi-algebraic) functions to construct a lift of size \(d\). This difference is due to the constructed factorization map being built from matrices of rank 1 only. 

In the commutative case, there are several known obstructions to the existence of spectrahedral lifts, see for instance \cite{fgpst}. In particular, there is a result in \cite{schei} which uses a globally positive non-sum-of-squares polynomial to conclude that a convex set is not a spectrahedrop. Using this result at level \(s = 1\) also obstructs the existence of a \(P^d\)-lift of an operator system, since a \(P^d\)-lift would in particular imply the existence of level-wise spectrahedral lifts. We generalise this obstruction by requiring only a matrix polynomial which is globally positive and not a sum-of-squares. The following example, found in \cite{hn}, shows that this can happen even if the degree and number of variables are such that all globally positive scalar polynomials would be sums-of-squares.

\begin{example}
    The matrix polynomial
    \[H(x,y,z) =
    \begin{pmatrix}
    	2z^2+x^2&-xy&-xz \\
    	-xy&2x^2+y^2&-yz \\
    	-xz&-yz&z^2+2y^2
    \end{pmatrix}
    \] 
    is globally positive semidefinite but not a sum-of-squares of matrix polynomials. To see this, let \(v = (a, b, c)^T\) be a vector of three variables. Then \(v^T H v\) is a well-known example of a globally positive scalar polynomial which is not a sum of squares of polynomials \cite{choi}. Hence, the matrix polynomial itself can also not be a sum of squares of matrix polynomials.
\end{example}

Motivated by such matrix polynomials, the following corollary is an adaptation of the obstruction in \cite{schei}. The proof is inspired by the approach in \cite{faw}.

\begin{corollary} \label{cor:obstruction}
	Let \(q_1, \dots, q_m \in \R[x]\) be polynomials in \(x = (x_1, \dots, x_n)\). Define \[q\colon \R^n \to \R^m; a \mapsto (q_1(a), \dots, q_m(a)),\] $C=\overline{\rm cc}(q(S))$, and  $V_t \coloneqq \operatorname{span}_{\operatorname{Her}_t(\C)}(q_1, \dots, q_m)\subseteq{\rm Her}_t(\R[x])$ for $t\geqslant 1.$
	Assume that for some \(t \in \N\) there exists a homogeneous globally positive semidefinite hermitian matrix polynomial \(H \in V_t\) which is not a sum-of- squares of matrix polynomials, and such that for all \(u \in \R^n\), the translate \(H_u(x) \coloneqq H(x - u)\) is also contained in \(V_t\).
	Then \(C^\textnormal{min}\) has no $P^d$-lift.
\end{corollary}

\begin{proof} 
	Assume for contradiction that \(C^\textnormal{min}\) has a $P^d$-lift. By \Cref{thm:NC_lift_iff_SOHS} we can find a $d^2$-dimensional space $W$ of real-valued semialgebraic functions on $C$ with the sums-of-squares-property.  We set \(X \coloneqq W\circ q\), the space of functions from $W$ combined with $q,$ which is a space of semialgebraic functions on $\R^n.$ Since semialgebraic functions are smooth almost everywhere, and the space is finite-dimensional, we find \(u \in \R^n\) such that all occuring functions are smooth on a small neighbourhood around \(u\). Since we can shift everything to $0$ and back, we can assume  \(u=0\) without loss of generality.  	  
  
	Every globally positive semidefinite matrix polynomial \(P \in V_t\) is  of the form \(P = \phi \circ q\) for some \(\phi \in \operatorname{Pos}(C, \operatorname{Psd}_t(\C))\). We thus obtain  \[P(a) = \sum_k P_k(a)^* P_k(a)\] for all $a\in\R^n$, where $P_k\in{\rm Mat}_{d,t}(X).$
 
	We now proceed analogously to \cite{faw}: By assumption we have a homogeneous \(H \in V_t\) which is globally positive, so \(\deg(H) = 2N\), but not a sum-of-squares of matrix polynomials. By the above, we write it as a sum-of-squares of matrices with entries from \(X\). As argued before, they are all smooth around 0, so we perform a Taylor expansion of the matrix-valued functions. Comparing the lowest-degree homogeneous terms on both sides shows that \(H(x)\) is a sum-of-squares of matrix polynomials, a contradiction to our assumptions. Hence, \(C^{\rm min}\) does not have a \(P^d\)-lift.
\end{proof}

\begin{remark}
    We note that in \cite{faw}, the assumption of \(H\) being homogeneous can be dropped at the expense of a somewhat more technical proof, though the precise assumptions are different. A similar improvement of \Cref{cor:obstruction} might hold, possibly also requiring an adjustment of the assumptions.
\end{remark}

\newpage\thispagestyle{empty}
\addcontentsline{toc}{part}{Bibliography} 
{\linespread{1}\bibliographystyle{dpbib} \bibliography{references}}
\end{document}